\documentclass[12pt]{amsart}

\usepackage{amsmath,amssymb,amsthm}
\usepackage{amsfonts}
\usepackage[mathscr]{eucal}
\usepackage{color}
\newtheorem{theorem}{Theorem}[section]
\newtheorem{lemma}[theorem]{Lemma}

\theoremstyle{definition}

\newtheorem{example}[theorem]{Example}

\theoremstyle{remark}
\newtheorem{remark}[theorem]{Remark}

\numberwithin{equation}{section}

\newcommand{\rd}{{\mathbb R^d}}

\newcommand{\rr}{{\mathbb R}}

\def\E{{\mathbb E}}
\def\P{{\mathbb P}}

\def\D{{\mathbb D}}

\newcommand{\convfd}{\stackrel{f.d.}{\Longrightarrow}}

\begin{document}

\title{\bf The fractional Poisson process and the inverse stable subordinator}

\author{Mark M. Meerschaert}
\address{Mark M. Meerschaert, Department of Statistics and Probability,
Michigan State University, East Lansing, MI 48823}
\email{mcubed@stt.msu.edu}
\urladdr{http://www.stt.msu.edu/$\sim$mcubed/}
\thanks{The research of MMM was partially supported by NSF grants DMS-0803360 and EAR-0823965.}

\author{Erkan Nane}
\address{Erkan Nane, Department of Mathematics and Statistics, 221 Parker Hall, Auburn University, Auburn, Alabama 36849, USA}
\email{nane@auburn.edu}

\author{P. Vellaisamy}
\address{P. Vellaisamy, Department of Mathematics,
 Indian Institute of Technology Bombay, Powai, Mumbai 400076, INDIA.}
\email{pv@math.iitb.ac.in}

\begin{abstract}
The fractional Poisson process is a renewal process with Mittag-Leffler waiting times.  Its distributions solve
a time-fractional analogue of the Kolmogorov forward equation for a Poisson process.  This paper shows that a
traditional Poisson process, with the time variable replaced by an independent inverse stable subordinator, is also a
 fractional Poisson process.  This result unifies the two main approaches in the stochastic theory of time-fractional
 diffusion equations.   The equivalence extends to a broad class of renewal processes that include models for tempered
  fractional diffusion, and distributed-order (e.g., ultraslow) fractional diffusion.  The paper also establishes an interesting connection between the fractional Poisson process and Brownian time.
\end{abstract}


\maketitle

\section{Introduction}

The fractional Poisson process (FPP) was introduced and studied
by Repin and Saichev \cite{RS00}, Jumarie \cite{Jumarie01}, Laskin \cite{Laskin03}, Mainardi et al.
\cite{MGS04, MGV07}, Uchaikin et al. \cite{UCS08} and Beghin and Orsingher \cite{BO09, BO10}.  The FPP is a
natural generalization of the usual Poisson process, with an
interesting connection to fractional calculus.  This renewal process has IID waiting
times $J_n$ that satisfy
\begin{equation}\label{MLwait}
\P(J_n>t)=E_{\beta}(-\lambda t^\beta)
\end{equation}
for $0<\beta\leq 1$, where
\begin{equation}\label{MLdef}
E_\beta(z)=\sum_{k=0}^{\infty}\frac{z^k}{\Gamma (1+\beta k)} \cdot
\end{equation}
denotes the Mittag-Leffler function.  When $\beta=1$, the waiting
times are exponential with rate $\lambda$, since $e^z=E_1(z)$. Let
$T_n=J_1+\cdots+J_n$ be the time of the $n$th jump.  Then the FPP
\begin{equation}\label{renewal-process}
N_\beta(t)=\max\{n\geq 0: T_n\leq t\}
\end{equation}
is a renewal process with Mittag-Leffler waiting times.

A compound FPP is
obtained by subordinating a random walk to the FPP.  The resulting
process is non-Markovian (unless $\beta=1$) and the distribution
of that process solves a ``master equation'' analogous to the
Kolmogorov equation for Markov processes, with the usual integer
order time derivative replaced by a fractional derivative.

The continuous time random walk (CTRW) is another useful model in fractional calculus.  Consider a CTRW whose IID
particle jumps $Y_n$
have PDF $w(x)$, and whose IID waiting times $(J_n)$ are
Mittag-Leffler variables independent of $(Y_n)$.  The
particle location after $n$ jumps is $S(n)=Y_1+\cdots+Y_n$, and the
CTRW $S(N_\beta(t))$ gives the particle location at time $t\geq
0$. Hilfer and Anton \cite{hilfer-anton95}  show that the PDF
$p(x,t)$ of the CTRW $S(N_\beta(t))$ solves the fractional master
equation
\begin{equation}\label{master-equation-1}
\partial_t^\beta p(x,t)=-\lambda p(x,t)+ \lambda\int_{-\infty}^\infty p(x-y,t)w(y)\,dy
\end{equation}
where $\partial_t^\beta$ denotes the Caputo fractional derivative.
The Caputo fractional derivative, defined for $0\leq n-1<\beta<n$
by
\begin{equation}\label{CaputoDef}
\partial_t^\beta g(t)=\frac{1}{\Gamma(n-\beta)}\int_0^t {(t-r)^{n-1-\beta}} g^{(n)}(r)dr,
\end{equation}
where $g^{(k)}$ denotes the $k$-th derivative of $g$, was invented
to properly handle initial values \cite{Caputo}.

If $\beta=1$, then $\partial_t^\beta$ is the usual first derivative.  The corresponding CTRW $S(N_1(t))$ is a
compound Poisson process, and \eqref{master-equation-1} reduces to
\begin{equation}\label{cpcp}
\partial_t p(x,t)=-\lambda p(x,t)+ \lambda \int_{-\infty}^\infty p(x-y,t)w(y)\,dy ,
\end{equation}
the Cauchy problem associated with this infinitely divisible L\'evy process.  Then a general result on Cauchy
problems \cite[Theorem 3.1]{fracCauchy} implies that the PDF of the time-changed process $S(N_1(E(t)))$ solves
the fractional Cauchy problem \eqref{master-equation-1}, where
\begin{equation}\label{EtDef}
E(t)=\inf\{r>0:D(r)>t\}
\end{equation}
is the right-continuous inverse (hitting time, first passage time) of $D(t)$, a standard $\beta$-stable subordinator
with $\E[e^{-sD(t)}]=e^{-ts^\beta}$ for some $0<\beta<1$.

Since the PDF of both $S(N_\beta(t))$ and $S(N_1(E(t)))$ solve the same governing equation \eqref{master-equation-1},
with the same point-source initial condition (i.e., both processes start at the origin), these two processes have the same one dimensional distributions.   Heuristically, the degenerate case $Y_n\equiv 1$ gives $S(n)=n$, which strongly suggest that the FPP $N_\beta(t)$ and the process $N_1(E(t))$ have the same one dimensional distributions.  We will call $N_1(E(t))$ the fractal time Poisson process (FTPP), since it comes from a self-similar time change (see, e.g., \cite[Proposition 3.1]{limitCTRW}).  In this paper, we will prove that the FPP and the FTPP are in fact the same process, by showing that the waiting times between jumps in the FTPP are IID Mittag-Leffler.  This strong connection between the FPP and the FTPP unifies the two main approaches in the stochastic theory of fractional diffusion.  For example, the FPP approach was used recently in the work of Behgin and Orsingher \cite{BO09}, while the inverse stable subordinator is a key ingredient in \cite{MNV}.

\section{Two equivalent formulations}
Recall that the fractional Poisson process (FPP) $N_\beta(t)$ is a
renewal process with Mittag-Leffler waiting times \eqref{MLwait},
and the fractal time Poisson process (FTPP) $N_1(E(t))$ is Poisson
process, with rate $\lambda>0$, time-changed via the inverse
stable subordinator \eqref{EtDef}. The proof that the FPP and the
FTPP are the same process requires the following simple lemma.

\begin{lemma}\label{Dlemma}
Let $D(t)$ be a strictly increasing right-continuous process with left-hand limits, and let $E(t)$ be its right-continuous inverse defined by \eqref{EtDef}.  Then
\begin{equation}\label{eq2.1}
D(r-)=\sup\{t>0:E(t)<r\}
\end{equation}
for any $r>0$.
\end{lemma}

\begin{proof}
Let $t_0=\sup\{t>0:E(t)<r\}$.  Then there exists a sequence of points $t_n\uparrow t_0$ such that $E(t_n)<r$ for all
 $n$.  Let $\varepsilon_n =r-E(t_n)>0$.  If $r>E(t)$ then, since $D(t)$ is strictly increasing, $D(r)>t$.
 Since $D(r)$ is right-continuous, it follows that $D(E(t))\geq t$ for all $t>0$.  Then we have
  $t_n\leq D(E(t_n))=D(r-\varepsilon_n)< D(r-)$.  Letting $n\to\infty$ shows that  $D(r-)\geq t_0$.

Since $D$ has left-hand limits, for any $r_n\uparrow r$ we have
$D(r_n)\to D(r-)$ as $n\to\infty$.  If $D(r-)> t_0$, then for some
$r_n<r$ we have $D(r_n)>t_0$.  Since $E(t)$ is nondecreasing and
continuous, this implies that $E(D(r_n))\geq r$, by definition of
$t_0$.  But, $E(D(r))=r$ for all $r>0$ implying that $r_n\geq r$,
which is a contradiction. Thus, \eqref{eq2.1} follows.
\end{proof}

\begin{theorem}\label{th1}
For any $0<\beta<1$, the FTPP $N_1(E(t))$ is also a FPP.  That is, the waiting times between jumps of the FTPP are
IID Mittag-Leffler.
\end{theorem}

\begin{proof}
Let $W_n$ be an IID sequence with $\P(W_n>t)=e^{-\lambda t}$ and
$V_n=W_1+\cdots +W_n$ so that the Poisson process
$N_1(t)=\max\{n\geq 0:V_n\leq t\}$.  Let
\begin{equation}\label{eq2.2}
\tau_n=\sup\{t>0:N_1(E(t))<n\}
\end{equation}
 denote the jump
times of the FTPP. This definition of the jump times takes into
account the fact that $E(t)$ has
 constant intervals corresponding to the jumps of the process $D(t)$.  Using the fact that $\{N_1(t)< n\}=\{V_n> t\}$ for the Poisson process, along with \eqref{eq2.2}, we have
\[\tau_n=\sup\{t>0:E(t)<V_n\} .\]
Then Lemma \ref{Dlemma} implies that $\tau_n=D(V_n-)$.  Define $X_1=\tau_1$ and $X_n=\tau_n-\tau_{n-1}$ for $n\geq 2$,
 the waiting times between jumps of the FTPP.  In order to show that the FTPP is an FPP, it suffices to show that $X_n$
  are IID Mittag-Leffler, i.e., they are IID with $J_n$.

Recall that the Laplace transform of the exponential distribution
$\E(e^{-sW_n})
={\lambda}/(\lambda+s)$.
Also recall that $\E(e^{-sD(t)})=e^{-ts^\beta}$. Since $D(t)$ is a
L\'evy process, it has no fixed points of discontinuity and hence
$D(t-),D(t)$ are identically distributed for all $t\geq 0$.
(Indeed, $D(t)=D(t-) ~a.s. $ \cite [Lemma 2.3.2]{appm}).

 Then a conditioning argument yields
\begin{equation}\begin{split}\label{eq2.3}
\E(e^{-s\tau_1})=\E(e^{-sD(W_1-)})&=\E\left[\E\left(e^{-sD(W_1-)}\big|W_1\right)\right]\\&=\E\left[\E\left(e^{-sD(W_1)}
\big|W_1\right)\right]= \E\left[e^{-W_1
s^\beta}\right]=\frac{\lambda}{\lambda+s^\beta} \cdot
\end{split}\end{equation}
Let $f_\beta(x)=\partial_x[1-E_\beta(-\lambda x^\beta)]$
be the Mittag-Leffler PDF of $J_n$.  It is well known that
\[\int_0^\infty e^{-sx} E_\beta(-\lambda x^\beta)\,dx=\frac{s^{\beta-1}}{\lambda+s^\beta},\]
see for example \cite[Eq.\ (3.4)]{MNV}.  Now integrate by parts to see that
\begin{equation}\begin{split}\label{T1LT}
\E(e^{-sT_1})&=\int_0^\infty e^{-sx} f_\beta(x)\,dx\\
&=\int_0^\infty s e^{-sx} \left(1-E_\beta(-\lambda x^\beta)\right)\,dx\\
&=s\left[\frac 1s - \frac{s^{\beta-1}}{\lambda+s^\beta}\right]
=\frac{\lambda}{\lambda+s^\beta}=\E(e^{-s\tau_1})
\end{split}\end{equation}
and then the uniqueness theorem for LT implies that $T_1,\tau_1$ are identically distributed.  In particular,
 $X_1$ has the same Mittag-Leffler distribution as $J_1$.

A straightforward extension of this argument shows that
$(T_1,\ldots,T_n)$ is identically distributed with
$(\tau_1,\ldots,\tau_n)$ for any positive integer $n$.  To ease
notation, we only write the case $n=2$.  First observe that
\[\E(e^{-s_1T_1}e^{-s_2T_2})=\E(e^{-s_1J_1}e^{-s_2(J_1+J_2)})= \frac{\lambda}{\lambda+(s_1+s_2)^\beta}
\cdot\frac{\lambda}{\lambda+s_2^\beta}, \] using the independence
of $J_1$ and $J_2$. Next write
\begin{equation*}\begin{split}
\E(e^{-s_1D(t_1)}e^{-s_2D(t_1+t_2)})
&=\E(e^{-s_1D(t_1)}e^{-s_2[D(t_1)+D(t_1+t_2)-D(t_1)]})\\
&=\E(e^{-(s_1+s_2)D(t_1)}e^{-s_2[D(t_1+t_2)-D(t_1)]})\\
&=e^{-t_1(s_1+s_2)^\beta}e^{-t_2s_2^\beta},
\end{split}\end{equation*}
using the fact that $D(t)$ has independent increments.  Then
\begin{equation*}\begin{split}
\E(e^{-s_1\tau_1-s_2\tau_2})
&=\E(e^{-s_1D(W_1-)-s_2D([W_1+W_2]-)})\\
&=\E\left[\E\left(e^{-s_1D(W_1)-s_2D(W_1+W_2)}\big|W_1,W_2\right)\right]\\
&=\E\left[e^{-W_1(s_1+s_2)^\beta}e^{-W_2s_2^\beta}\right]
=\frac{\lambda}{\lambda+(s_1+s_2)^\beta}\cdot\frac{\lambda}{\lambda+s_2^\beta}=\E(e^{-s_1T_1}e^{-s_2T_2}) .
\end{split}\end{equation*}
Now an application of the continuous mapping theorem shows that $(J_1,\ldots,J_n)$ is identically distributed
with $(X_1,\ldots,X_n)$ for any positive integer $n$.  Then $(X_n)$ is an IID sequence, so $N_1(E(t))$ is a renewal
process.
\end{proof}

\begin{remark}
Theorem \ref{th1} extends a result in Behgin and Orsingher \cite{BO09}.  They define (in our notation) a random variable $E(t)$ and show that the two random variables $N_\beta(t)$ and $N_1(E(t))$ have the same density function, by comparing their Laplace transforms.  They identify $E(t)$ only though its density function, which they express in terms of an integral involving the density of $D(t)$, see Remark \ref{BOremark} for more detail.   Cahoy, Uchaikin,  Woyczynski \cite{cahoy2010} also connect the Mittag-Leffler distribution with a stable law.  They note that (in our notation) $P(J_n>t)=E[\exp(-\lambda t^\beta/D(1)^\beta)]$, which is useful in simulations.  To connect this with our work, note that $E(t) = (t/D(1))^\beta$ in distribution (see Corollary 3.1 in \cite{limitCTRW}), so that $P(J_n>t)=E[\exp(-\lambda E(t))]$.  A result of Bingham \cite{bingham} shows that the the Laplace transform of the stable hitting time $E(t)$ is Mittag-Leffler, so that \eqref{MLwait} holds.
\end{remark}

\begin{remark}
The proof of Theorem \ref{th1} uses the fact that, if $D(t)$ is a $\beta$-stable subordinator and $W_1$ is exponential, then $D(W_1)$ has a Mittag-Leffler distribution.  This fact was first noticed by Pillai \cite{Pillai}, who showed that
$W_1^{1/\beta}D(1)$ is Mittag-Leffler.  These are equivalent because $D(t)$ is identically distributed with $t^{1/\beta}D(1)$.  This Mittag-Leffler distribution is also known as the positive Linnik law, e.g., see Huillet \cite{Huillet}.   It has the property of geometric stability:  A geometric random sum of  Mittag-Leffler variables is again Mittag-Leffler, e.g., see Kozubowski \cite{Kozubowski94}.
\end{remark}

Next we want to show that the FTPP $N_1(E(t))$, and hence also the
FPP $N_\beta(t)$, occurs naturally as a CTRW scaling limit.  This provides a further justification for the FPP as a robust physical model, see for example Laskin \cite{Laskin03}.
Suppose now that $\P(J_n>t)=t^{-\beta}L(t)$, where $0<\beta<1$ and
$L$ is slowly varying.  For example, this is true of the
Mittag-Leffler waiting times.  Then $J_1$ belongs to the strict
domain of attraction of some stable law $D$ with index
$0<\beta<1$, i.e., there exist $b_n>0$ such that
\begin{equation}\label{doabeta}
b_n(J_1+\dots+J_n)\Rightarrow D,
\end{equation}
where $D(1)=D>0$ almost surely, and  $\Rightarrow$ denotes convergence in distribution.  Let $b(t)=b_{[t]}$. Then
 $b(t)=t^{-1/\beta}L_0(t)$ for some slowly varying function $L_0(t)$ (e.g., see \cite[XVII.5]{feller}).  Since $b$
 varies regularly with index $-1/\beta$, $b^{-1}$ is
regularly varying with index $1/\beta>0$ and so by  \cite[Property
1.5.5]{seneta} there exists a regularly varying function
 $\tilde b$ with index $\beta$ such that $1/b(\tilde b(c))\sim c$, as $c\to\infty$. Here we use the notation $f\sim g$
 for positive functions $f,g$ if and only if $f(c)/g(c)\to 1$ as $c\to\infty$.  Let $T_n=J_1+\cdots+J_n$ and define a
 renewal process
\begin{equation}\label{RtDef}
R(t)=\max\{n\geq 0:T_n\leq t\}
\end{equation}
with these waiting times.  Next, construct a CTRW with iid
Bernoulli jumps $Y^{(p)}_n$ with $\P(Y^{(p)}_n=1)=p$ and
$\P(Y^{(p)}_n=0)=1-p$, independent of $(J_n)$.  Let
$S^{(p)}(n)=Y^{(p)}_1+\cdots+Y^{(p)}_n$, a binomial random
variable.  Then $S^{(p)}(R(t))$ is a CTRW with heavy tailed
waiting times and Bernoulli jumps.

\begin{theorem}\label{th2}
The FTPP is the process limit of a CTRW sequence:
\begin{equation}\label{CTRWtoFTPP}
\big\{S^{(1/\tilde b(c))}([\lambda R(ct)])\big\}_{t\geq 0}\Rightarrow \big\{N_1(E(t))\big\}_{t\geq 0}
\end{equation}
as $c\to\infty$ in the $M_1$ topology on $D([0,\infty),\rr)$.
\end{theorem}

\begin{proof}
Since the sequence $(J_n)$ is in the strict domain of attraction of a $\beta$-stable random variable $D$,
 \cite[Corollary 3.4]{limitCTRW} shows that
\[ \big\{\tilde b(c)^{-1}R(ct)\big\}_{t\geq 0}\Rightarrow\big\{E(t)\big\}_{t\geq 0}\quad\text{as $c\to\infty$.}\]
in the Skorokhod $J_1$ topology, where $D(t)$ is the stable subordinator with $D(1)=D$, and $E(t)$ is given by
 \eqref{EtDef}.

Since the binomial random variable $S^{(p)}(n)$ has LT $\E(e^{-sS^{(p)}(n)})=(1+(e^{-s}-1)p)^n$ for any $n\geq 0$, it
 follows that
\[\E(e^{-sS^{(p)}([\lambda t/p])})=(1+(e^{-s}-1)p)^{[\lambda t/p]}\to \exp(-\lambda
t(1-e^{-s})),\]
as $p\to 0$, using the fact that $(1+ap)^{1/p}\to e^a$ as $p\to 0$.  It follows by the continuity theorem for LT that
 $S^{(p)}([\lambda t/p])\Rightarrow N_1(t)$ for any $t>0$, since $\exp(-\lambda t(1-e^{-s}))$ is the LT of the Poisson
  random variable $N_1(t)$.  Then a standard argument (e.g., see \cite[Example 11.2.18]{RVbook} shows that we also get
\[ \big\{S^{(p)}([\lambda t/p])\big\}_{t\geq 0}\convfd \big\{N_1(t)\big\}_{t\geq 0}
,\]
as $p\to 0$, where $\convfd$ denotes convergence of all finite
dimensional distributions.  Since the sample paths of
 $S^{(p)}([\lambda t/p])$ are increasing and $N_1(t)$ is continuous in probability, being a L\'evy process, $J_1$
 convergence follows using \cite[Theorem 3]{bingham}.

Since the CTRW waiting times $(J_n)$ are independent of the jumps $(Y^{(p)}_n)$, and since $1/\tilde b(c)\to 0$ as
 $c\to\infty$, it follows that
$$
(S^{(1/\tilde b(c))}([\lambda t \tilde b(c)]),\tilde b(c)^{-1}R(ct))\Rightarrow (N_1(t),E(t))
$$
in the $J_1$ topology of the product space $D([0,\infty),\rr\times \rr)$, by \cite[Theorem 3.2]{billingsley}.
Since the process $E(t)$ is nondecreasing and continuous, \cite[Theorem 13.2.4]{Whittbook} along with the continuous mapping theorem yields
\[S^{(1/\tilde b(c))}([\lambda R(ct)])=S^{(1/\tilde b(c))}([\lambda\cdot \tilde b(c)^{-1} R(ct)\cdot  \tilde b(c)])\Rightarrow N_1(E(t))\]
in the $M_1$ topology on $D([0,\infty),\rr)$.
\end{proof}

\begin{remark}
For the specific case of Mittag-Leffler waiting times, where $\P(J_n>t)=E_{\beta}(-t^\beta)$, we can take $b_n=n^{-1/\beta}$ in \eqref{doabeta}.  To check this, note that
\begin{equation*}\begin{split}
\E(e^{-sb_nT_n})&=\left(\frac{1}{1+(sb_n)^\beta}\right)^n
=\left(1-\frac{s^\beta}{n+s^\beta}\right)^n
\rightarrow e^{-s^\beta}=\E(e^{-sD(1)})
\end{split}\end{equation*}
as $n\to\infty$.
Then $\tilde b(c)=c^\beta$ and the CTRW convergence
\eqref{CTRWtoFTPP} reduces to $S^{(c^{-\beta})}([\lambda
R(ct)])\Rightarrow N_1(E(t))$ as $c\to\infty$.  Substitute
$p=c^{-\beta}$ to get
\begin{equation}
S^{(p)}([\lambda R(p^{-1/\beta}t)])\Rightarrow N_1(E(t)),
\quad\text{as $p\to 0$.}
\end{equation}
\end{remark}

\section{Fractional calculus}\label{FCsec}
This section develops some interesting connections between the fractional Poisson process and fractional calculus.  In the process, some apparent inconsistencies in the existing literature will be explained.
Behgin and Orsingher \cite[Eq.\ (2.17)]{BO09} show that the FPP of order $0<\beta<1$ has distribution
\begin{equation}\label{BOpmf}
\P(N_\beta(t)=k)=\int_0^\infty e^{-\lambda x}\frac{(\lambda
x)^k}{k!} V(x,t)\,dx,
\end{equation}
where $V(x,t)$ is a ``folded PDF'' defined on $x>0$, for each
$t>0$, by $V(x,t)=2v(x,t)$, and $v(x,t)$ is another PDF with
$x\in\rr$ for each $t>0$ that solves
\begin{equation}\begin{split}\label{BOfde}
\partial_t^{2\beta} v(x,t)&=\partial_x^2 v(x,t);\\
v(x,0)&=\delta(x);\\
\partial_t v(x,0)&\equiv 0, \quad\text{if $1/2<\beta<1$.}
\end{split}\end{equation}
It is also stated in \cite[Eq.\ (1.9)]{BO09} that the FPP
$N_\beta(t)=N_1(T_t)$, where $T_t$ is a random process with PDF
$V(x,t)$ for $t>0$. However, that process is identified only in
terms of its one dimensional distributions (PDF).  Theorem
\ref{th1} shows that the inverse stable subordinator $E(t)$ is one
such process.

On the other hand, a simple conditioning argument shows that the
equivalent FTPP process has distribution
\begin{equation}\label{MMMpmf}
\P(N_1(E(t))=k)=\int_0^\infty P(N_1(x)=k)h(x,t)\,dx=\int_0^\infty e^{-\lambda x}\frac{(\lambda x)^k}{k!}h(x,t)\,dx
\end{equation}
where $h(x,t)$ is the density of $E(t)$, a PDF on $x>0$ for each $t>0$.  It follows from \cite[Theorem 4.1]{triCTRW}
that this PDF solves
\begin{equation}\label{MMMfde}
\partial_t^\beta h(x,t)=-\partial_x h(x,t);\quad\text{$h(x,0)=\delta(x)$.}
\end{equation}
In view of Theorem \ref{th1}, the two distributions \eqref{BOpmf} and \eqref{MMMpmf} must be equal.  Thus, the main
purpose of this section is to reconcile the two fractional differential equations \eqref{BOfde} and \eqref{MMMfde}.

\begin{theorem}\label{th3}
Let $N_\beta(t)$ be a fractional Poisson process \eqref{renewal-process} with $0<\beta<1$, so that \eqref{BOpmf} holds.  Let $N_1(E(t))$ be the equivalent fractal time Poisson process, where $E(t)$ is the standard inverse $\beta$-stable subordinator with PDF $h(x,t)$, so that \eqref{MMMpmf} holds.  Then
\begin{equation}\label{BOeqMMM}
h(x,t)=2v(x,t) \quad\text{for all $x>0$ and $t>0$.}
\end{equation}
In particular, the two fractional partial differential equations \eqref{BOfde} and \eqref{MMMfde} are consistent, in
the sense that the folded solution $V(x,t)=2v(x,t)$ to \eqref{BOfde} coincides with the solution $h(x,t)$
to \eqref{MMMfde}.
\end{theorem}

\begin{proof}
Mainardi \cite[Eq.\ (3.2)]{Mainardi96} shows that the solution to the fractional diffusion-wave equation \eqref{BOfde}
has LT
\begin{equation}\label{BOLT}
\tilde v(x,s)=\int_0^\infty e^{-st}v(x,t)dt=\frac 12 s^{\beta-1}e^{-|x|s^\beta}
\end{equation}
while \cite[Eq. (3.13)]{triCTRW} shows that
\begin{equation}\label{MMMLT}
\tilde h(x,s)=s^{\beta-1}e^{-xs^\beta} .
\end{equation}
Since both are differentiable in $t$, they are also continuous, so LT uniqueness for continuous functions implies
\eqref{BOeqMMM}.

Take Fourier transforms in \eqref{BOLT} to see that the solution to \eqref{BOfde} has Fourier-Laplace transform (FLT)
\begin{equation}\label{BOFLT}
\bar v(k,s)=\int_0^\infty e^{-st}\int_{-\infty}^\infty e^{-ikx}v(x,t)dxdt =\frac{s^{2\beta-1}}{s^{2\beta}+k^2} ,
\end{equation}
where we have used the fact that $e^{-a|x|}$ has FT $2a/(a^2+k^2)$.  Rearrange to get
\[s^{2\beta} \bar v(k,s)-s^{2\beta-1}=-k^2 \bar v(k,s) \]
and invert the FT to get
\[s^{2\beta} \tilde v(x,s)-s^{2\beta-1} v(x,0)=\partial_x^2 \tilde v(x,s),
\]
using the fact that $\partial_x f(x)$ has FT $(ik)\hat f(k)$ and $v(x,0)=\delta(x)$ has FT $\hat v(k,0)\equiv 1$.
To invert the LT, note that the Caputo fractional derivative $\partial_t^\beta f(t)$ has
LT $s^\beta \tilde f(s)-s^{\beta-1}f(0)$ if $0<\beta\leq 1$, and LT $s^\beta \tilde
f(s)-s^{\beta-1}f(0)-s^{\beta-2}f'(0)$ if $1<\beta\leq 2$.  This is easy to verify from the definition
\eqref{CaputoDef}, using the corresponding formula for the integer derivative, along with the fact that $s^{\beta-1}$
is the LT of $t^{-\beta}/\Gamma(1-\beta)$.  Now use the remaining initial condition $\partial_t v(x,0)\equiv 0$
for $1/2<\beta<1$ to invert the LT, and arrive at \eqref{BOfde}.

Likewise, the solution to \eqref{MMMfde} has FLT
\begin{equation}\label{MMMFLT}
\bar h(k,s)=\int_0^\infty e^{-ikx} s^{\beta-1}e^{-xs^\beta} dx =\frac{s^{\beta-1}}{s^\beta+ik} ,
\end{equation}
using the fact that $e^{ax}I(x\geq 0)$ has FT $1/(a+ik)$.  To see that these are consistent, compute the FLT of $h(|x|,t)$:
\begin{equation*}\begin{split}
\int_0^\infty e^{-st}\int_{-\infty}^\infty e^{-ikx} h(|x|,t)dxdt
&=\int_0^\infty e^{-st}\left(\int_0^\infty e^{-ikx} h(x,t)dx+\int_0^\infty e^{ikx} h(x,t)dx\right)dt\\
&=\frac{s^{\beta-1}}{s^\beta+ik}+\frac{s^{\beta-1}}{s^\beta-ik}
=2\left(\frac{s^{2\beta-1}}{s^{2\beta}+k^2}\right)
=2\bar v(k,s) .
\end{split}\end{equation*}
Invert the FLT to see that $h(|x|,t)=2v(x,t)$ for all $x\in\rr$ and $t>0$.  To verify the LFT solution, take FT in \eqref{MMMfde} to get
\[\partial_t^\beta \hat h(k,t)=-ik\, \hat h(x,t) \]
and apply the LT to get $s^\beta \bar h(k,s)-s^{\beta-1}=-ik\, \bar h(k,s)$, using the point source initial condition $\hat h(k,0)\equiv 1$.
\end{proof}

\begin{remark}\label{BOremark}
Behgin and Orsingher \cite[Eq.\ (2.21)]{BO09} show that
\[v(x,t)=\frac 1{2\Gamma(1-\beta)}\int_0^t (t-w)^{-\beta}p(|x|,t)dx\]
where $p(x,t)$ is the density of the stable subordinator $D(t)$, while \cite[Theorem 3.1]{triCTRW} implies that
\[h(x,t)=\frac 1{\Gamma(1-\beta)}\int_0^t (t-w)^{-\beta}p(x,t)dx .\]
This gives another proof that $h(|x|,t)=2v(x,t)$.
\end{remark}

\begin{remark}\label{FPPpmf}
A closed form expression for the probability mass function $p(n,t)=\P(N_1(E(t))=n)=P(N_\beta(t)=n)$ follows easily from \eqref{MMMpmf}.  Use \cite[Eq. (3.13)]{triCTRW} to write
\[\tilde p(n,s)=\int e^{-st}p(n,t)dt=\int_0^\infty e^{-\lambda x}\frac{(\lambda x)^n}{n!}s^{\beta-1}e^{-xs^\beta}\,dx \]
and use the formula for the gamma density to compute
\begin{equation}\label{pnsLT}
\tilde p(n,s)=\frac{s^{\beta-1}}{\lambda+s^\beta}\frac{\lambda^n}{(\lambda +s^\beta)^n} .
\end{equation}
Invert using the generalized Mittag-Leffler function
\[E^\gamma_{\alpha,\theta}(z)=\sum_{r=0}^\infty\frac{(\gamma)_r z^r}{r!\Gamma(\alpha r+\theta)}\]
where $(\gamma)_r=\gamma (\gamma+1)\cdots(\gamma+r-1)$ is the Pochammer Symbol.  Formula (2.5) of \cite{prabhakar} gives
$$
 \int_0^\infty e^{-st} t^{\gamma-1}E^{\delta}_{\nu,\gamma}(\omega t^{\nu})dt=\frac{s^{\nu \delta-\gamma}}{(s^{\nu}-\omega)^{\delta}} .
 $$
Substitute $\nu=\beta$, $\delta=n+1$ and $\gamma=\beta n +1$ to get
\begin{equation}\label{pntPMF}
 p(n,t)=(\lambda t^{\beta})^n E^{n+1}_{\beta, \beta n+1}(-\lambda t^{\beta})=\frac{(\lambda t^{\beta})^n}{n!}\sum_{r=0}^\infty\frac{(n+r)!}{r!}\frac{(-\lambda t^\beta)^r}{\Gamma(\beta (r+n)+1)}
\end{equation}
which is the same form obtained by Jumarie \cite{Jumarie01}, Laskin \cite{Laskin03},  Beghin and Orsingher \cite{BO09, BO10} and Cahoy \cite{cahoy} using different methods.
\end{remark}

The equivalence in Theorem \ref{th3} results from folding the solution to the fractional diffusion-wave equation
\eqref{BOfde}.  Another fractional partial differential equation for the density $h(x,t)$ of the standard inverse
$\beta$-stable subordinator $E(t)$, which is closer to the form \eqref{BOfde}, can be obtained by arguments similar
to those used in \cite{bmn-07} to connect the inverse stable subordinator to iterated Brownian motion.  In that theory,
it is customary to avoid distributions by imposing a functional initial condition.

\begin{theorem}\label{heat-type-pde-1}
Let $E(t)$ be the standard inverse $\beta$-stable subordinator
with density $h(x,t)$. Then for any $f\in L_2(\rr)\cap C^1(\rr)$,
the function
\begin{equation}\label{uxtdef}
u(x,t)=\E_x[f(E(t))]=\int_0^\infty f(x+y)h(y,t)dy
\end{equation}
solves the fractional differential equation
\begin{equation}\label{BTh2}
\partial_t^{2\beta} u(x,t) = -\partial_x f(x)
\frac{t^{\beta-1}}{\Gamma(1-\beta)}+ \partial_x^{2}u(x,t); \quad u(x,0) = f(x).
\end{equation}
In particular, when $\beta=1/2$, \eqref{uxtdef} solves
\begin{equation}\label{BTh}
\partial_t u(x,t) =
\frac{-\partial_x f(x)}{\sqrt{\pi t}}+ \partial_x^{2}u(x,t); \quad u(0,x) = f(x) ,
\end{equation}
and in this case we also have $u(x,t)=\E_x[f(|B(t)|)]$, where
$B(t)$ is a Brownian motion with variance $2t$.
\end{theorem}

\begin{proof}
From \eqref{MMMFLT}, we have
\begin{equation*}\begin{split}
\bar u(k,s)&=\frac{s^{\beta-1}\hat f(k)}{s^{\beta}+ik}
=\frac{s^{\beta-1}\hat f(k)}{s^{\beta}+ik}\cdot\frac{s^{\beta}-ik}{s^{\beta}-ik}
=\frac{s^{2\beta-1}-iks^{\beta-1}}{s^{2\beta}+k^2}\, \hat f(k)
\end{split}\end{equation*}
so that $s^{2\beta} \bar u(k,s)-s^{2\beta-1}\hat f(k)=-ik \hat
f(k) s^{\beta-1} -k^2 \bar u(k,s)$, which inverts to \eqref{BTh2}.
It is well known that the Brownian motion first passage time
$D(y)=\inf\{t>0:B(t)>y\}$ is a stable subordinator with index
$\beta=1/2$ \cite[Example 1.3.19]{appm}.
Then it is easy to see that
\[E(t)=\inf\{y>0:D(y)>t\}=\sup\{B(r):0\leq r\leq t\}\]
and this recovers the fact, typically proven using the reflection principle, that
\[P(E(t)> y)=2P(B(t)> y).\]
Then $E(t)$ and $|B(t)|$ have the same one dimensional distributions, so we also have $u(x,t)=\E_x[f(|B(t)|)]$.  Note that $X(t)=X(0)-t$ is {\it a fortiori} a continuous Markov process associated with the shift semigroup $T(t)f(x)=\E_x[f(X(t))]=f(x-t)$ with generator
\[L_xf(x)=\lim_{t\to 0+}\frac {T(t)f(x)-f(x)}{t}
=-\partial_x f(x) .\]
Then \cite[Corollary 3.4]{bmn-07} implies that $u(x,t)$ solves the equation
\[
\partial_t u(x,t) =
\frac{{L_x}f(x)}{\sqrt{\pi t}}+ {L_x}^{2}u(x,t); \quad u(0,x) = f(x) .
\]
When $L_x=-\partial_x$, this reduces to \eqref{BTh}, a special case of \eqref{BTh2} with $\Gamma(1/2)=\sqrt{\pi}$.
\end{proof}

\begin{remark}
In the case $f(x)=\delta(x)$, Theorem \ref{heat-type-pde-1} gives an alternative governing equation for $h(x,t)$.
Note that \eqref{BTh2} is very similar to the governing equation \eqref{BOfde} for the unfolded PDF.
\end{remark}

\begin{remark}
The process $|B(t)|$ in Theorem \ref{heat-type-pde-1} is not the
same process as the inverse $1/2$-stable subordinator $E(t)$ in
Theorem \ref{th1}, although they have the same one dimensional
distributions.  Hence, the FTPP is not the same as the Brownian
time Poisson process $N_1(|B(t)|)$.  However, we do have
$E(t)=\sup\{B(r):0\leq r\leq t\}$, so that a Poisson process
subordinated to the supremum of a Brownian motion is an FPP with
$\beta=1/2$.
\end{remark}

\begin{remark}
Let $E(t)$ be the standard inverse stable subordinator of index $\beta=1/m$ for integer $m>1$. Then
 \cite[Remark 3.11]{bmn-07}, \cite[Theorem 1.1]{nane} and Keyantuo and Lizama \cite[Theorem 3.3]{lizama-keyantuo}
 imply that $u(x,t)=\E_x[f(E(t))]$ solves
\[
\partial_t u(x,t) =
\sum_{j=1}^{m-1} \frac{t^{j/m-1}}{\Gamma(j/m)} (-\partial_x)^j
f(x) + (-\partial_x)^m u(x,t);\quad   u(0,x) =  f(x),
\]
for $t>0$ and $x\in \rr$,  which is then equivalent to
\eqref{MMMfde}.  The proof is similar to Theorem
\ref{heat-type-pde-1}. For example, when $\beta=1/3$ use
\begin{equation*}\begin{split}
\bar{u}(s,k)
&=\frac{s^{-2/3}\hat{f}(k)}{s^{1/3}+ik}\cdot
\frac{s^{2/3}-s^{1/3}ik+k^2}{s^{2/3}-s^{1/3}ik+k^2}
= \frac{1-s^{-1/3}ik+s^{-2/3}k^2 }{s+ik^3}\,\hat{f}(k) .
\end{split}\end{equation*}
\end{remark}

\section{Renewal processes and inverse subordinators}\label{sec4}
Theorem \ref{th1} shows that a Poisson process, time-changed by an inverse stable subordinator, yields a renewal process
 with Mittag-Leffler waiting times.  This section extends that result to arbitrary subordinators that are strictly increasing.  Let $D(t)$ be a
strictly increasing L\'evy process (subordinator) with ${\mathbb E}[e^{-s D(t)}]=e^{-t\psi_D(s)}$, where the Laplace exponent
\begin{equation}\label{psiD2}
\psi_D(s)=bs+\int_0^\infty(e^{-s x}-1)\phi_D(dx) ,
\end{equation}
$b\geq 0$, and $\phi_D$ is the L\'evy measure of $D$.  Then we must have either
\begin{equation}\label{A1}
\phi_D(0,\infty)=\infty ,
\end{equation}
or $b>0$, or both.  Let $E(t)$ be the inverse subordinator \eqref{EtDef}, and recall that $N_1(t)$ is a Poisson process with rate $\lambda$.

\begin{theorem}\label{th4.1}
The time-changed Poisson process $N_1(E(t))$ is a renewal process whose IID waiting times $(J_n)$ satisfy
\begin{equation}\label{Jdef2}
\P(J_n>t)=\E [e^{-\lambda E(t)}] .
\end{equation}
\end{theorem}

\begin{proof}
The proof is similar to Theorem \ref{th1}.  Take
$N_1(t)=\max\{n\geq 0:V_n\leq t\}$, where $V_n=W_1+\cdots +W_n$,
with $W_n$ IID as $\P(W_n>t)=e^{-\lambda t}$. Let
\[\tau_n=\sup\{t>0:N_1(E(t))<n\}=\sup\{t>0:E(t)<V_n\} \]
and apply Lemma \ref{Dlemma} to get $\tau_n=D(V_n-)$.  Then, as
in the proof of Theorem \ref{th1}, we have
\begin{equation}\begin{split}\label{tauLT}
\E(e^{-s\tau_1})&=\E(e^{-sD(W_1-)})\\
&=\E\left[\E\left(e^{-sD(W_1)}\big|W_1\right)\right]=\E\left[e^{-W_1 \psi_D(s)}\right]=\frac{\lambda}{\lambda+\psi_D(s)} .
\end{split}\end{equation}
By \cite[Corollary 3.5]{triCTRW}, the IID random variables $J_n$ in \eqref{Jdef2} satisfy
\begin{equation}\label{EtLLT}
\int_0^\infty e^{-st}\, \P(J_n>t)\, dt=\int_0^\infty e^{-st}\, \E [e^{-\lambda E(t)}]\, dt
=\frac{\psi_D(s)}{s(\lambda+\psi_D(s))} .
\end{equation}
Integrate by parts to get
\begin{equation}\label{JnLT}
\int_0^\infty e^{-st}\, \P_{J_n}(dt)=\int_0^\infty s e^{-st}\, \left[1-\P(J_n>t)\right]\, dt
=1-\frac{\psi_D(s)}{\lambda+\psi_D(s)}=\frac{\lambda}{\lambda+\psi_D(s)},
\end{equation}
 which shows that $T_1=J_1$ is identically distributed with $\tau_1$.  Extend this argument, as in the
 proof of Theorem \ref{th1}, to show that $(T_1,\ldots,T_n)$ is identically distributed with $(\tau_1,\ldots,\tau_n)$
 for any positive integer $n$.  For example, when $n=2$, write
 \begin{equation*}\begin{split}
\E(e^{-s_1D(t_1)}e^{-s_2D(t_1+t_2)})
&=\E(e^{-(s_1+s_2)D(t_1)}e^{-s_2[D(t_1+t_2)-D(t_1)]})
=e^{-t_1\psi_D(s_1+s_2)}e^{-t_2\psi_D(s_2)}
\end{split}\end{equation*}
and condition to get
\begin{equation*}\begin{split}
\E(e^{-s_1\tau_1-s_2\tau_2})
&=\E(e^{-s_1D(W_1-)-s_2D([W_1+W_2]-)})\\
&=\E\left[\E\left(e^{-s_1D(W_1)-s_2D(W_1+W_2)}\big|W_1,W_2\right)\right]\\
&=\E\left[e^{-W_1\psi_D(s_1+s_2)}e^{-W_2\psi_D(s_2)}\right]
=\frac{\lambda}{\lambda+\psi_D(s_1+s_2)}\cdot\frac{\lambda}{\lambda+\psi_D(s_2)}.
\end{split}\end{equation*}
On the other hand,
\[\E(e^{-s_1T_1}e^{-s_2T_2})=\E(e^{-s_1J_1}e^{-s_2(J_1+J_2)})= \frac{\lambda}{\lambda+\psi_D(s_1+s_2)}\cdot\frac{\lambda}{\lambda+\psi_D(s_2)} \]
using the fact that $(J_n)$ are IID.  To finish the proof, use continuous mapping to show that $(J_1,\ldots,J_n)$ is identically distributed with $(X_1,\ldots,X_n)$, where $X_n=\tau_n-\tau_{n-1}$ are the waiting times between jumps for the process $N_1(E(t))$.
\end{proof}

\begin{remark}
Let $N_D(t)$ denote the renewal process from Theorem
\ref{th4.1}, so that
\begin{equation}\label{renewal-process-D}
N_D(t)=\max\{n\geq 0: T_n\leq t\},
\end{equation}
where $T_n=\sum_{i=1}^n J_i$ and $(J_n)$ are IID according to \eqref{Jdef2}.  Theorem \ref{th4.1} shows that
 $N_D(t)=N_1(E(t))$.  This extends the relation $N_\beta(t)=N_1(E(t))$ from Theorem \ref{th1}, the special case
  of an inverse stable subordinator $E(t)$ and Mittag-Leffler waiting times $J_n$, to a general inverse
  subordinator.
\end{remark}

\begin{remark}
Let $M(t)=\E (N_D(t))$ denote the renewal function
of the renewal process $N_D(t)$. Then using Lageras \cite[Equation
4]{Lageras}, it follows that the LT of $M(t)$ is
$\lambda/\psi_D(s).$
\end{remark}

\begin{remark}
A simple conditioning argument shows that
\[
p_D(n,t)=\P(N_D(t)=n)=\int_0^\infty P(N_1(x)=n)h(x,t)\,dx=\int_0^\infty e^{-\lambda x}\frac{(\lambda x)^n}{n!}h(x,t)\,dx
\]
where $h(x,t)$ is the density of $E(t)$.  A straightforward extension of the argument in Remark \ref{FPPpmf} shows that
\begin{equation}\label{genLT}
\tilde p_D(n,s)=\frac{s^{-1}\psi_D(s)}{\lambda+\psi_D(s)}\cdot\frac{\lambda^n}{(\lambda +\psi_D(s))^n}
\end{equation}
which reduces to \eqref{pnsLT} in the special case $\psi_D(s)=s^\beta$ for a stable subordinator $D(t)$.  Use \eqref{EtLLT} and \eqref{JnLT} to see that the first factor in \eqref{genLT} is the LT ($t\mapsto s$) of  $\tilde h(\lambda,t)=\P(J_n>t)=\E(e^{-\lambda E(t)})$, and the second is the LT of $T_n=J_1+J_2+\cdots+J_n$ with $J_n$ IID as in \eqref{Jdef2}. Denote the distribution of $T_n$ by $F^{(n,*)}$, the $n$-fold convolution of the distribution function $F$ of $J_1$.  Invert the LT to get
\begin{equation}\label{pDnt}
p_D(n,t)=\int_{0}^t \tilde h(\lambda,t-s) F^{(n,*)}(ds)
\end{equation}
which extends \eqref{pntPMF}.
\end{remark}

\section{CTRW scaling limits and governing equations}

In this section, we extend the fractional calculus results of Section \ref{FCsec} to the inverse subordinators of
Section \ref{sec4}.
A general theory of CTRW scaling limits and governing equations is developed in \cite{triCTRW}.  Consider a
sequence of CTRW indexed by a scale parameter $c>0$. Take $J_n^{c}$ nonnegative IID random variables representing the
 waiting times between particle jumps and $T^{c}(n)=\sum_{i=1}^nJ_i^{c}$, the time of the $n$th jump.
  Let $Y_i^{c}$ be IID random vectors on $\rd$ representing the particle jumps, independent of the waiting times, and
   set $S^{c}(n)=\sum_{i=1}^nY_i^{c}$, the location of the particle after $n$ jumps. Define $N_t^{c}=\max\{n\geq 0 : T^{c}(n)\leq t\}$,
   the number of jumps by time $t\geq 0$ and
\begin{equation}\label{defCTRW}
X^{c}(t)=S^{c}(N_t^{c})=\sum_{i=1}^{N_t^{c}}Y_i^{c}
\end{equation}
the position of the particle at time $t\geq 0$ and scale $c>0$.  Assume a triangular array limit
\begin{equation}\label{convassump}
\{(S^{c}(ct),T^{c}(ct))\}_{t\geq
0}\Rightarrow\{(A(t),D(t)\}_{t\geq 0},\quad\text{as $c\to\infty$},
\end{equation}
in the $J_1$ topology on $D([0,\infty),\rd\times\rr_+)$, so that
$A(t)$ and $D(t)$ are independent L\'evy processes on $\rd$ and
$\rr$,
 respectively. Since the waiting times are nonnegative, $D(t)$ is
a subordinator.  In this section, we assume the drift $b=0$ in
\eqref{psiD2}, as well as condition \eqref{A1} and
\begin{equation}\label{A2}
\int_0^1 y|\ln y|\,\phi_D(dy)<\infty .
\end{equation}
Assumption \eqref{A1} implies that the process $\{D(t)\}$ is strictly increasing, i.e., $D(t)$ is not compound Poisson.  Then \cite[Theorem 3.1]{triCTRW} shows that the inverse subordinator $E(t)$ in \eqref{EtDef} has a Lebesgue density
\begin{equation}\label{Edensity}
h(x,t)=\int_0^t\phi_D(t-y,\infty)\,\P_{D(x)}(dy) .
\end{equation}
Write $\E[e^{-sD(t)}]=e^{-t\psi_D(s)}$, as before.  Let
$P(x,t)=\P(A(t)\leq x)$ be the distribution function of
 $A(t)$, and write
\[\hat P(k,t)=\int e^{-ik\cdot x}P(dx,t)=e^{-t\psi_A(k)},\]
where $\psi_A(k)$ is the Fourier symbol of $A$.  The symbols define pseudo-differential operators:
$\psi_D(\partial_t)f(t)$ has LT $\psi_D(s)\tilde f(s)$, and $\psi_A(-iD_x)f(x)$ has FT $\psi_A(k)\hat f(k)$, for
suitable functions $f$.  Then \cite[Theorem 2.1]{triCTRW} establishes the CTRW scaling limit
\begin{equation}\label{conv1}
\{X^{c}(t)\}_{t\geq 0}\Rightarrow \{A(E(t))\}_{t\geq
0},\quad\text{as $c\to\infty$},
\end{equation}
in the $M_1$-topology on $D([0,\infty),\rd)$.  Recall that a function $Q$ is a {\it mild solution} to a space-time
pseudo-differential equation if its (Fourier-Laplace or Laplace-Laplace) transform solves the equivalent algebraic
equation in transform space.   The next result is a small extension of \cite[Theorem 4.1]{triCTRW}.

\begin{theorem}\label{gov}
Assume \eqref{convassump} holds, where $D(t)$ is a subordinator without drift such that conditions
 \eqref{A1} and \eqref{A2} hold.  The distribution function of the CTRW limit process $A(E(t))$ in \eqref{conv1} is
 given by
\begin{equation}\label{h1}
Q(x,t)=\int_0^\infty P(x,u)h(u,t)\,du
\end{equation}
where $h(u,t)$ is the density \eqref{Edensity} of the inverse subordinator $E(t)$.  The distribution function
 $Q(x,t)$ solves the generalized Cauchy problem
\begin{equation}\label{gCp}
\psi_D(\partial_t)Q(x,t)=-\psi_A(-iD_x)Q(x,t)+H(x)\phi_D(t,\infty)
\end{equation}
in the mild sense, where $H(x)=I(x\geq 0)$ is the Heaviside function.  Furthermore, $P(x,u)$ solves the Cauchy problem
\begin{equation}\label{CP}
\partial_t P(x,t)=-\psi_A(-iD_x) P(x,t);\quad P(x,0)=H(x),
\end{equation}
and $h(x,t)$ solves the inhomogeneous Cauchy problem
\begin{equation}\label{ICP}
\partial_x h(x,t)=-\psi_D(\partial_t) h(x,t) +\delta(x)\phi_D(t,\infty) .
\end{equation}
\end{theorem}

\begin{proof}
The proof is similar to \cite[Theorem 4.1]{triCTRW}.  Equation
\eqref{h1} follows from a simple conditioning argument.  Apply
\cite[Theorem 3.6]{triCTRW} to see that $Q(x,t)$ has FLT
\begin{equation}\label{barm}
\bar Q(k,s)=\int_0^\infty e^{-st}\int_{\rd} e^{-ik\cdot x}Q(dx,t)\,dt=\frac 1s\frac{\psi_D(s)}{\psi_A(k)+\psi_D(s)}
\end{equation}
and rearrange to get
\begin{equation}\label{MFLT}
\psi_D(s)\,\bar Q(k,s)=-\psi_A(k)\,\bar Q(k,s)+s^{-1}\psi_D(s) .
\end{equation}
  From \cite[Eq.\ (3.12)]{triCTRW} we get
\begin{equation}\label{bert}
\int_0^\infty e^{-su}\phi_D(u,\infty)\,du =s^{-1}\psi_D(s) .
\end{equation}
Now invert the FLT \eqref{MFLT}, using \eqref{bert} and $\int e^{-ik\cdot x}H(dx)\equiv 1$, to arrive at \eqref{gCp}.  It is well known that $P(x,t)$ solves the Cauchy problem \eqref{CP}, see for example \cite{HiPh}.  Equation \eqref{EtLLT} shows that the bivariate Laplace transform (LLT)
\[\tilde h(\lambda,s)=\int_0^\infty \int_0^\infty e^{-\lambda z-st}h(z,t)\,dt\,dz=\frac 1s\frac{\psi_D(s)}{\lambda+\psi_D(s)} . \]
This rearranges to
\[\lambda \tilde h(\xi,s)=-\psi_D(s)\tilde h(\lambda,s) +s^{-1}\psi_D(s) .\]
Inverting the LLT using \eqref{bert} to see that $h(x,t)$ solves \eqref{ICP}.
\end{proof}

For any random walk $S(n)=\sum_{i=1}^n Y_i$, the compound Poisson process $A(t)=S(N_1(t))$ is a L\'evy process.  Introduce IID waiting times \eqref{Jdef2} between these random walk jumps to get a CTRW.  In this case, the CTRW is exactly of the form $A(E(t))$, without passing to the limit.  Then the governing equations in Theorem \ref{gov} pertain to the CTRW itself.

\begin{theorem}\label{govCTRW}
Assume  $D(t)$ is a subordinator without drift such that conditions \eqref{A1} and \eqref{A2} hold, and let $E(t)$ be
 the inverse subordinator \eqref{EtDef}.  Take $J_n$ IID waiting times according to \eqref{Jdef2}, and let $N_D(t)$
 denote the renewal process \eqref{renewal-process-D}.  Take $Y_n$ IID jumps on $\rd$, independent from $(J_n)$, with
 common distribution $\mu$, and let $S(n)=\sum_{i=1}^n Y_i$.  Then the distribution function $P(x,t)=\P(X(t)\leq x)$
 of the CTRW $X(t)=S(N_D(t))$ solves the generalized Cauchy problem
\begin{equation}\label{gCpCTRW}
\psi_D(\partial_t)P(x,t)=-\lambda P(x,t)+\lambda \int P(x-y,t)\,\mu(dy)+H(x)\phi_D(t,\infty)
\end{equation}
in the mild sense.  Furthermore, $X(t)=A(E(t))$, where
$A(t)=S(N_1(t))$ is a compound Poisson process.
\end{theorem}

\begin{proof}
Theorem \ref{th4.1} yields $N_D(t)=N_1(E(t))$, and then the CTRW is
\[X(t)=S(N_D(t))=S(N_1(E(t)))=A(E(t)) .\]
A standard conditioning argument shows that the compound Poisson
FT $\hat P(k,t)=e^{-t\psi_A(k)}$, where the Fourier symbol
$\psi_A(k)=\lambda (1-\hat\mu(k))$.  The inverse FT of
$\psi_A(k)\hat f(k)$ is
\begin{equation}\label{CPsymbol}
\psi_A(-iD_x)f(x)=-\lambda f(x)+\lambda \int f(x-y)\,\mu(dy)
\end{equation}
using the FT convolution property.  Now Theorem \ref{gov} implies that \eqref{gCpCTRW} holds.
\end{proof}

\begin{remark}\label{rem5.3}
In the situation of Theorem \ref{govCTRW}, where $A(t)$ is compound Poisson, the distribution function $P(x,t)=\P(A(t)\leq x)$ solves the Cauchy problem \eqref{CP}, which can be written in this case as
\begin{equation}\label{KFE}
\partial_t P(x,t)=-\lambda P(x,t)+\lambda \int_{-\infty}^\infty P(x-y,t)\,\mu(dy);\quad P(x,0)=H(x) .
\end{equation}
This is the Kolmogorov forward equation for the Markov process $A(t)$.  If $\mu$ has density $w(x)$, apply
$\partial_x$ on both sides of \eqref{KFE} to see that the probability density $p(x,t)=\partial_x P(x,t)$ of $A(t)$
 solves \eqref{cpcp}.  If $D$ is the stable subordinator with Laplace symbol $\psi_D(s)=s^\beta$, then
  \eqref{gCpCTRW} holds with $\phi_D(t,\infty)=t^{-\beta}/\Gamma(1-\beta)$ and $\psi_D(\partial_t)=\D_t^\beta$,
  the Riemann-Liouville fractional derivative.  The Riemann-Liouville fractional derivative is defined
  for $0\leq n-1<\beta<n$ by
\begin{equation}\label{RLDef}
\D_t^\beta g(t)=\frac{1}{\Gamma(n-\beta)}\frac{d^n}{dt^n} \int_0^t
{(t-r)^{n-1-\beta}}{g^{(n)}(r)\,dr},
\end{equation}
which differs from the Caputo derivative \eqref{CaputoDef} in that the derivative is applied after the integration.
 The LT of $\D_t^\beta g(t)$ is $s^\beta \tilde g(s)$.  Apply $\partial_x$ to both sides of \eqref{gCpCTRW} in this
 case to get
\[\D_t^\beta p(x,t)=-\lambda p(x,t)+\lambda \int p(x-y,t)\,\mu(dy) +\delta(x)\frac{t^{-\beta}}{\Gamma(1-\beta)} ,\]
the fractional kinetic equation of Zaslavsky \cite{Zaslavsky}.  To recover \eqref{master-equation-1},
 use $\partial_t^\beta g(t)=\D_t^\beta g(t)-g(0)t^{-\beta}/\Gamma(1-\beta)$ and $p(x,0)=\delta(x)$.
\end{remark}

\begin{remark}\label{OBrem}
In the special case where $\mu=\varepsilon_1$ is a point mass, so that $Y_n=1$ almost surely, $A(t)=N_1(t)$ is a Poisson process with rate $\lambda>0$.    Then the distribution function $P(x,t)$ of the renewal process $N_D(t)=A(E(t))$ solves
\begin{equation}\label{gCpRP}
\psi_D(\partial_t)P(x,t)=-\lambda [P(x,t)- P(x-1,t)]+H(x)\phi_D(t,\infty) .
\end{equation}
If $D$ is the stable subordinator with Laplace symbol $\psi_D(s)=s^\beta$, Equation \eqref{gCpRP} reduces to
\[\partial_t^\beta P(x,t)=-\lambda [P(x,t)- P(x-1,t)] \]
as in Remark \ref{rem5.3}. The probability mass function $p(n,t)=P(n,1)-P(n-1,t)=\Delta P(n,t)$ for $n>0$.  Apply the difference operator $\Delta$ on both sides to obtain
\[\partial_t^\beta p(n,t)=-\lambda [p(n,t)- p(n-1,t)] \]
as in Jumarie \cite{Jumarie01}.
\end{remark}

\begin{remark}\label{CTRWrem}
Scher and Lax \cite{ScherLax} showed that a CTRW with waiting time distribution $\omega$ and jump distribution $\nu$
has FLT
\[\bar Q(k,s)=\frac 1s \frac{1-\tilde\omega(s)}{1-\tilde\omega(s)\hat\nu(k)} ,
\]
where $\hat\nu(k)=\int e^{-ik\cdot x}\nu(dx)$.  To reconcile with Theorem \ref{govCTRW}, recall from \eqref{JnLT} that
 the waiting times \eqref{Jdef2} in Theorem \ref{govCTRW} have LT
\[\tilde\omega(s)=\int e^{-st}\omega(dt)=\frac{\lambda}{\lambda+\psi_D(s)}\]
and then it follows that
$\psi_D(s)=\lambda({1-\tilde\omega(s)})/{\tilde\omega(s)} .$
The jumps $Y_n$ in Theorem \ref{govCTRW} have Fourier symbol $\psi_A(k)=\lambda (1-\hat\mu(k))$ and then \eqref{barm}
implies
\begin{equation*}\begin{split}
\bar Q(k,s)&=\frac 1s\frac{\psi_D(s)}{\psi_A(k)+\psi_D(s)}
=\frac 1s \frac{\frac{1-\tilde\omega(s)}{\tilde\omega(s)}}{\frac{1-\tilde\omega(s)}{\tilde\omega(s)}+(1-\hat\mu(k))}
=\frac 1s \frac{1-\tilde\omega(s)}{1-\tilde\omega(s)\hat\mu(k)}
\end{split}\end{equation*}
which provides a different proof that the CTRW equals $A(E(t))$ in this case.  To simulate the sample paths of the non-Markovian process $A(E(t))$, it is sufficient to simulate the CTRW.  In particular, the renewal process $N_D(t)$ gives the exact jump times of the inverse subordinator $E(t)$.
\end{remark}

\begin{remark}
In the general case, where $A(t)$ is not compound Poisson, Theorem \ref{govCTRW} provides a useful approximation.  Given a L\'evy process $A(t)$, take $Y_n=A(n)-A(n-1)$, so that $S(n)=A(n)$.  Take $N(t)$ a Poisson process with rate $1$, so that
$S(\lambda^{-1}N(\lambda t))$ is compound Poisson with Fourier symbol
\[\lambda(1-e^{-\lambda^{-1}\psi_A(k)})\to \psi_A(k),\quad\text{as $\lambda\to\infty$.}\]
Then $S(\lambda^{-1}N(\lambda t))\Rightarrow A(t)$ as $\lambda\to\infty$, and the CTRW with IID waiting times \eqref{Jdef2} and these compound Poisson jumps converges to $A(E(t))$ as $\lambda\to\infty$.  As in Remark \ref{CTRWrem}, this fact can be used to simulate sample paths of the process $A(E(t))$.   This fact has been exploited by Fulger, Scalas and Germano \cite{FSG} to develop fast simulation methods for space-time fractional diffusion equations.
\end{remark}

\begin{example}
Tempered stable subordinators are theoretically interesting \cite{temperedLM,Rosinski} and practically useful
 \cite{Cont04,TemperedStable}. Take $D(t)$ tempered stable with Laplace symbol $\psi_D(s)=(s+a)^\beta-a^\beta$
 for $a>0$ and $0<\beta<1$, and let $E(t)$ be its inverse \eqref{EtDef}.  Theorem \ref{th4.1} shows that $N_1(E(t))$
  is a renewal process.  Let $(\tau_n)$ denote the arrival times of this renewal process, and use \eqref{tauLT} to get
\begin{equation*}\begin{split}
\E(e^{-s\tau_1})
&=\frac{\lambda}{\lambda+(s+a)^\beta-a^\beta} .
\end{split}\end{equation*}
This tempered fractional Poisson process $N_1(E(t))$ has tempered Mittag-Leffler waiting times, but with a different
 rate parameter:  Use \eqref{T1LT} to see that the  Mittag-Leffler PDF
$f(t)=\partial_t[1-E_\beta(-\eta t^\beta)]$
has Laplace transform
${\eta}/({\eta+s^\beta})$,
and so
\[\int_0^\infty e^{-st}f(t)e^{-at} dt=\frac{\eta}{\eta+(s+a)^\beta} .\]
Of course $f(t)e^{-at}$ is not a PDF, and in fact we have (set $s=0$ above)
\[\int_0^\infty f(t)e^{-at} dt=\frac{\eta}{\eta+a^\beta} .\]
Then the tempered Mittag-Leffler PDF
$f_a(t)=f(t)e^{-at}({\eta+a^\beta})/{\eta}$
has LT
\begin{equation*}\begin{split}
\int_0^\infty e^{-st}f_a(t) dt
&=\frac{\eta+a^\beta}{\eta+(s+a)^\beta}=\frac{\lambda}{\lambda+(s+a)^\beta-a^\beta}=\E(e^{-s\tau_1})
\end{split}\end{equation*}
when $\eta+a^\beta=\lambda$.  Cartea and Del-Castillo
\cite{Cartea2007} show that the tempered fractional
 derivative
$\psi_D(\partial_t)g(t)=e^{-a t}\,{\partial_t^{\beta}}[e^{a
t}\,g(t)]-a^\beta g(t) .$ It is also known (e.g., see
\cite{temperedLM}) that the corresponding L\'evy measure is
exponentially tempered: $\psi_D(dt)=e^{-at} \psi(dt)$, where
$\psi(t,\infty)=t^{-\beta}/\Gamma(1-\beta)$ is the L\'evy measure
 of the standard $\beta$-stable subordinator.
Then Theorem \ref{govCTRW} shows that the CTRW with tempered Mittag-Leffler waiting times and compound Poisson
 jumps solves a tempered fractional Cauchy problem
\begin{equation*}\label{TSgov}
e^{-a t}\,{\partial_t^{\beta}}[e^{a t}\,P(x,t)]-a^\beta P(x,t)
=\psi_A(-iD_x) P(x,t)+H(x)\phi_D(t,\infty)
\end{equation*}
with $\psi_A(-iD_x)$ given by \eqref{CPsymbol} and $\phi_D(t,\infty)= {\beta}\int_t^\infty e^{-at}t^{-\beta-1}dt
 /{\Gamma(1-\beta)}$.  More generally, Theorem \ref{gov} shows that the distribution function of the CTRW scaling
 limit $A(E(t))$ is governed by this equation, with the corresponding operator $\psi_A(-iD_x)$.  Apply $\partial_x$
 on both sides of \eqref{gCpRP} to see that the PDF of the renewal process with tempered Mittag-Leffler waiting times
  solves
\begin{equation*}\label{TSFPPgov}
e^{-a t}\,{\partial_t^{\beta}}[e^{a t}\,p(x,t)]-a^\beta p(x,t)=-\lambda [p(x,t)- p(x-1,t)]+\delta(x)\phi_D(t,\infty) .
\end{equation*}
A wide variety of tempered stable models in $\rd$ are discussed in Rosi\'nski \cite{Rosinski}.  Random walks in $\rd$
 with tempered stable scaling limit are developed in \cite{TSconv}.  For exponentially tempered stable waiting times
  in $\rr^1$, a renewal process with tempered Mittag-Leffler waiting times gives the same process exactly,
  without taking limits.  This can be useful for simulating sample paths.
\end{example}

\begin{example}
Chechkin et al.\ \cite{klafterultra,retarding} used distributed
order fractional derivatives to model multi-scale anomalous
subdiffusion, where a different power law pertains at short and
long time scales, and ultraslow diffusion, for a plume of
particles spreading at a logarithmic rate.  Given a finite Borel
measure $\nu$ on $(0,1)$, the distributed order fractional
derivative is defined by
\begin{equation}\label{DOFDdef}
\D_t^{\nu}g(t)=\int_0^1 \partial_t^\beta g(t)\nu(d\beta),
\end{equation}
where $\partial_t^\beta$ is the Caputo fractional derivative \eqref{CaputoDef}.
If $\nu$ is discrete, this is a linear combination of fractional derivatives.  Let $D(t)$ be the distributed
order stable subordinator with Laplace symbol $\psi_D(s)=\int s^\beta \nu(d\beta)$ and $E(t)$ its inverse \eqref{EtDef}.
Let $\nu(d\beta)=p(\beta)d\beta$ for some $p\in C^1(0,1)$, then by (2.19) in Kochubei \cite{koch3}
\begin{equation}\label{DOtail}
P(J_n>t)=\E (e^{\lambda E(t)})=\frac{\lambda}{\pi}\int_0^\infty r^{-1}e^{-tr}\Phi(r,1)dr
\end{equation}
where
$$
\Phi(r,1)=\frac{\int_0^1r^\beta\sin
(\beta\pi)\Gamma(1-\beta)p(\beta)d\beta}{[\int_0^1r^\beta \cos
(\beta\pi)\Gamma(1-\beta)p(\beta)d\beta+\lambda]^2+[\int_0^1r^\beta\sin
(\beta\pi) \Gamma(1-\beta)p(\beta)d\beta]^2}.
$$
Substitute \eqref{DOtail} into \eqref{pDnt} to obtain an explicit formula for the probability mass function of the distributed order Poisson process.

If $\nu(d\beta)=p(\beta)d\beta$, where $p(\beta)$ is regularly varying at $\beta=0$ with index
   $\alpha-1$ for some $\alpha>0$, then $\psi_D(s)=R(\log s)$ and $R$ is regularly varying at
    infinity with index $-\alpha$, see \cite[Lemma 3.1]{M-S-ultra}.   Then $E(t)$ is ``ultraslow'' in that
     $\E(E(t)^\gamma)=S(\log t)$, where $S$ varies regularly with index $\gamma\alpha$ at infinity,
     by \cite[Theorem 3.9]{M-S-ultra}.  Take an IID sequence of mixing variables $(B_i)$ with distribution
      $\mu$ concentrated on $(0,1)$, and assume $\P(J_i^{c}>u|B_i=\beta)=c^{-1}u^{-\beta}$ for $u\geq
      c^{-1/\beta}$,
      so that the waiting times are conditionally Pareto.  Then \cite[Theorem 3.4]{M-S-ultra} implies that the
       distributed order stable subordinator is a random walk limit $\sum_{i=1}^{[ct]} J_i^{c}\Rightarrow D(t)$.
       This requires  $\int (1-\beta)^{-1} \mu(d\beta)<\infty$ so that $\nu(d\beta)=\Gamma(1-\beta)\mu(d\beta)$ is
        a finite measure.  An easy computation shows that the L\'evy measure $\phi_D(t,\infty)=\int_0^1 t^{-\beta}
        \nu(d\beta)/\Gamma(1-\beta)$.  Then Theorem \ref{gov} implies that a CTRW with these conditionally Pareto
        waiting times has a scaling limit $A(E(t))$ whose distribution $Q(x,t)$ solves the distributed-order
        fractional diffusion equation
\[\D_t^{\nu}Q(x,t)=-\psi_A(-iD_x)Q(x,t) . \]
If $A(t)$ is compound Poisson, Theorem \ref{govCTRW} shows that
the distribution function $P(x,t)$ of a CTRW with waiting times
\eqref{Jdef2} solves \[\D_t^{\nu}P(x,t)=-\lambda P(x,t)+\lambda
\int P(x-y,t)\,\mu(dy),\] without passing to the limit.  Then the
PDF $p(x,t)$ of the renewal process with waiting times
\eqref{Jdef2} solves
\[\D_t^{\nu} p(x,t)=-\lambda [p(x,t)- p(x-1,t)] .\]
\end{example}


\begin{thebibliography}{99}

\bibitem{appm} Applebaum, D. (2009). {\it Levy Processes and
Stochastic Calculus}. Second Edition, Cambridge University Press,
New York.

\bibitem{fracCauchy} Baeumer, B. and Meerschaert, M. M. (2001). Stochastic solutions for fractional Cauchy problems.
{\it Fractional Calculus and Applied Analysis} {\bf 4}, 481--500.

\bibitem{bmn-07} Baeumer, B., Meerschaert, M. M. and Nane, E. (2009).
Brownian subordinators and fractional Cauchy problems.
\emph{Trans. Amer. Math. Soc.} {\bf 361} 3915--3930.

\bibitem{temperedLM} Baeumer, B. and Meerschaert, M. M. (2010). Tempered stable L\'evy motion and transient
super-diffusion.  {\it J. Comput. Appl. Math.} {\bf 233}, 2438--2448.

\bibitem{BO09} Beghin, L. and Orsingher, E. (2009). Fractional Poisson processes and related random motions.
{\it Electronic. Journ. Prob.},  14, n.61, 1790--1826.

\bibitem{BO10} Beghin, L. and Orsingher, E. (2010). Poisson-type processes governed by fractional and higher-order recursive differential equations. {\it Electronic. Journ. Prob.}, 15, 684--709.


\bibitem{billingsley}
Billingsley, P. (1968). {\it Convergence of Probability Measures.}
John Wiley, New York.

\bibitem{bingham}
 Bingham, N. H. (1971). Limit theorems for occupation times of Markov processes.
{\it Z. Wahrsch. Verw. Gebiete} {\bf 17},  1--22.

\bibitem{cahoy}Cahoy, D. O. (2007). Fractional Poisson process in terms of alpha-stable densities. Thesis (Ph.D.)–Case Western Reserve University. 106 pp.

\bibitem{cahoy2010} Cahoy, D. O., V. V. Uchaikin and W. A. Woyczynski (2010) Parameter estimation for fractional Poisson processes.   {\it J. Statist. Plann. Inf.} {\bf 140}, 3106–-3120.

\bibitem{Caputo}  Caputo, M. (1967). Linear models of dissipation whose Q is almost frequency independent,
Part II. {\it Geophys. J. R. Astr. Soc.} {\bf 13} 529--539.

\bibitem{Cartea2007} Cartea, A. and Del-Castillo, N. D. (2007).  {Fluid limit of the
  continuous-time random walk with general {L\'evy} jump distribution
  functions}.  {\it Phys. Rev. E} {\bf 76}, 041105.

\bibitem{TSconv} Chakrabarty, A. and Meerschaert, M. M. (2010). Tempered stable laws as random walk limits.
Preprint available at  {\tt www.stt.msu.edu/$\sim$mcubed/TSconv.pdf}.

\bibitem{retarding} Chechkin, A. V., Gorenflo, R. and Sokolov, I. M. (2002).
Retarding subdiffusion and accelerating superdiffusion governed
by distributed-order fractional diffusion equations.  {\it Phys. Rev. E} {\bf 66}, 046129--046135.

\bibitem{klafterultra} Chechkin, A. V., Klafter, J. and Sokolov,
I. M. (2003). Fractional Fokker-Plank equation for ultraslow
kinetics. {\it Europhys. Lett.} {\bf 63}(3), 326--332.

\bibitem{Cont04} Cont, R. and Tankov, P. (2004). \newblock {Financial
modelling with jump processes}. \newblock Chapman \& Hall/CRC,
Boca Raton, Florida.

\bibitem{feller} Feller, W. (1971). {\it An Introduction to Probability Theory and Its Applications}.
Vol. II, 2nd Ed., Wiley, New York.

\bibitem{FSG} Fulger, D., Scalas, E. and Germano, G. (2008). Monte Carlo simulation of uncoupled continuous-time
random walks yielding a
stochastic solution of the space-time fractional diffusion equation. {\it Phys Rev E} {\bf 77}, 021122.

 \bibitem{hilfer-anton95} Hilfer, R. and Anton, L. (1995). Fractional master equations and fractal time random
walks, Phys. Rev. E 51, R848-–R851.

\bibitem{HiPh} Hille, E. and Phillips, R. S. (1957). {\it Functional Analysis and Semi-Groups.}
Amer. Math. Soc. Coll. Publ. {\bf 31}, American Mathematical Society, Providence.

\bibitem{Huillet} Huillet, T. (2000). On Linnik's continuous-time random walk. {\it J. Phys. A} {\bf 33}, 2631--2652.

\bibitem{Jumarie01}Jumarie, G. (2001). Fractional master equation: non-standard analysis
and Liouville-Riemann derivative. {\it Chaos Solitons Fractals},
{\bf 12}, 2577-–2587.

\bibitem{lizama-keyantuo} Keyantuo, V. and Lizama, C. (2009). On a connection between powers of operators and fractional Cauchy problems. Preprint available at  {\tt netlizama.usach.cl/ Keyantuo-Lizama(AMPA)(2009).PDF}.

\bibitem{koch3} Kochubei, A. N. (2008). Distributed order calculus and equations of ultraslow diffusion.
{\it J. Math. Anal. Appl.} {\bf 340} 252--281.


\bibitem{Kozubowski94} Kozubowski, T. J. (1994). The inner characterization of geometric stable laws.
{\it Statist. Decisions} {\bf 12}, 307--321.

\bibitem{Lageras} Lageras, A. N. (2005). A renewal-process-type
expression for the moments of inverse subordinators. {\it J. Appl.
Probab.}, {\bf 42, }, 1134--1144.

\bibitem{Laskin03} Laskin, N. (2003). Fractional Poisson process. {\it Commun. Nonlinear Sci. Numer. Simul.},
{\bf 8}, 201-–213.


\bibitem{MetzlerK} Metzler, R. and Klafter, J. (2000). The random walk's guide to anomalous diffusion:
A fractional dynamics approach. {\it Phys. Rep.} {\bf 339}, 1--77.

\bibitem{Mainardi96} Mainardi, F. (1996). The fundamental solutions for the fractional diffusion-wave equation.
{\it Appl. Math. Lett.} {\bf 9}(6), 23--28.

\bibitem{MGS04} Mainardi, F., Gorenflo, R. and Scalas, E. (2004). A fractional generalization of the Poisson processes.
{\it Vietnam Journ. Math.} {\bf 32}, 53--64.

\bibitem{MPG07} Mainardi, F., Pagnini, G. and  Gorenflo, R. (2007). Some
aspects of fractional diffusion equations of single and
distributed order. {\it Appl. Math. Comput.}, {\bf 187}, 295-–305.

\bibitem{MGV07} Mainardi, F., Gorenflo, R. and Vivoli, A. (2007).
Beyond the Poisson renewal process: A tutorial survey. {\it J.
Comput. Appl. Math.}, {\bf 205}, 725--735.

\bibitem{RVbook}
 Meerschaert, M. M. and Scheffler, H. P. (2001).
\newblock {\it Limit Distributions for Sums of Independent Random Vectors: Heavy Tails in Theory and Practice}.
\newblock Wiley Interscience, New York.

\bibitem{limitCTRW}  Meerschaert, M. M. and  Scheffler, H. P. (2004).
Limit theorems for continuous time random walks with infinite mean
waiting times. {\it J. Appl. Probab.} {\bf 41} 623--638.

\bibitem{M-S-ultra} Meerschaert, M. M. and  Scheffler, H. P. (2006).  Stochastic model for ultraslow
diffusion. {\it  Stochastic Processes  Appl.} {\bf 116} 1215--1235.

\bibitem{triCTRW} Meerschaert, M. M. and  Scheffler, H. P. (2008). Triangular array limits for continuous
time random walks. {\it Stochastic Processes Appl.}
 {\bf 118}  1606--1633.

\bibitem{TemperedStable} Meerschaert, M. M., Zhang, Y. and Baeumer, B. (2008). Tempered anomalous diffusion in
heterogeneous systems. {\it Geophys. Res. Lett.} {\bf 35}, L17403.

\bibitem{MNV} Meerschaert, M. M., Nane, E. and Vellaisamy, P. (2009). Fractional Cauchy problems on bounded
domains. {\it  Ann. Probab.}  {\bf  37} 979--1007.

\bibitem{nane}Nane, E. (2010). Stochastic solutions of a class of higher order Cauchy problems in $\rd$.
Stochastics and Dynamics (To Appear).

\bibitem{Pillai} Pillai, R. N. (1990). On Mittag-Leffer functions and related distributions.
{\it Ann. Inst. Statist. Math.} {\bf 42}, 157--161.

\bibitem{prabhakar} Prabhakar T.R. (1971). A singular integral equation with a generalized Mittag Leffler function
in the kernel. Yokohama Math. J. 19, 7–-15.

\bibitem{RS00}Repin, O. N. and Saichev, A. I. (2000).  Fractional Poisson law. {\it Radiophys.
and Quantum Electronics}, {\bf43},  738–-741.

\bibitem{Rosinski} Rosi\'nski, J. (2007). {Tempering stable processes}. {\it Stoch. Proc. Appl.} {\bf 117},  677--707.

\bibitem{scalas1}  Scalas, E. (2004). Five years of continuous-time random walks in econophysics.
{\it Proceedings of WEHIA 2004}  (A. Namatame,  ed.) Kyoto, 3--16.

\bibitem{ScherLax} Scher, H. and Lax, M. (1973). Stochastic transport in a disordered solid. I. Theory.
{\it Phys. Rev. B} {\bf 7}, 4491--4502.

\bibitem{seneta} Seneta, E. (1976). {\it Regularly Varying Functions}.
Lecture Notes in Mathematics {\bf 508}, Springer-Verlag, Berlin.

\bibitem{UCS08} Uchaikin, V. V., Cahoy, D. O. and Sibatov, R. T. (2008). Fractional
processes: from Poisson to branching one. {\it Internat. J. Bifur.
Chaos Appl. Sci. Engrg.},  {\bf 18}, 2717--2725.

\bibitem{Whittbook} Whitt, W. (2002). {\it Stochastic-Process Limits}. Springer, New York.

\bibitem{Zaslavsky} Zaslavsky, G. (1994). Fractional kinetic equation for Hamiltonian chaos. Chaotic advection, tracer
dynamics and turbulent dispersion. {\it Phys. D} {\bf 76}, 110--122.

\end{thebibliography}
\end{document}